\documentclass[a4paper,11pt]{amsart}
\usepackage{amsmath,amssymb,amsfonts,amsthm}
\usepackage{amssymb}
\usepackage{pgf}
\usepackage{todonotes}
\usepackage{epsfig}
\usepackage{mathrsfs}
\usepackage{verbatim}

\newtheorem{thm}{Theorem}[section]
\newtheorem{prop}{Proposition}[section]

\newtheorem{lma}{Lemma}[section]

\def\N{{\rm I\kern-0.16em N}}
\def\R{{\rm I\kern-0.16em R}}
\def\E{{\rm I\kern-0.16em E}}
\def\P{{\rm I\kern-0.16em P}}
\def\F{{\rm I\kern-0.16em F}}
\def\B{{\rm I\kern-0.16em B}}
\def\C{{\rm I\kern-0.46em C}}
\def\G{{\rm I\kern-0.50em G}}
\newcommand{\ud}{\mathrm{d}}
\newcommand{\Q}{\mathbb{Q}}
\newcommand{\ov}[1]{\overline{#1}}

\numberwithin{equation}{section}
\usepackage{ae}
\def\ind{\mathrel{\hbox{\rlap{%
\hbox to 7.5pt{\hrulefill}}\raise6.6pt\hbox{\eka\char'167}}}}
\begin{document}
\title[Pricing European options on multiple assets]{Note on multidimensional Breeden-Litzenberger representation for state price densities}

\author[Talponen and Viitasaari]{Jarno Talponen \and Lauri Viitasaari}
\address{Department of Physics and Mathematics, University of Eastern Finland, P.O. Box 111, 80101 JOENSUU, talponen@iki.fi}
\address{Department of Mathematics and System Analysis, Helsinki University of Technology\\
P.O. Box 11100, FIN-00076 Aalto,  FINLAND} 

\begin{abstract}
In this note, we consider European options of type \\
$h(X^1_T, X^2_T,\ldots, X^n_T)$ depending on several underlying assets. We give a multidimensional version of the result of Breeden and Litzenberger \cite{Breeden} on the relation between derivatives of the call price and the risk-neutral density of the underlying asset. The pricing measure is assumed to be
absolutely continuous with respect to the Lebesgue measure on the state space.
 \end{abstract}


\maketitle

\section{Introduction}
Option valuation is one of the most central problems in financial mathematics. However, in many models of interest the option valuation
cannot be solved in closed form and thus different approaches have been developed. For instance one can use partial differential
equations (PDE) or partial integro-differential (PIDE) methods, Monte Carlo methods, or tree methods (see e.g. \cite{Jeanblanc}, \cite{Oksendal} and \cite{Fengler}). 
One approach to value complicated structured products is to determine their values in terms of the values of simple derivatives of the underlying such as call options and digital options. This is essentially static hedging.
In the work of Breeden and Litzenberger \cite{Breeden} the authors showed that if the second derivative of the call option price $V^C(K)$ with respect to the strike exists and is continuous, then the price of European option with payoff $f(X_T)$ is given by
\begin{equation}
\label{BL}
V^f = \int_{0}^\infty f(a)\frac{\ud^2}{\ud a^2}V^C(a)\ud a
\end{equation}
where we treated the short interest rate as $0$ for the sake of simplicity.
Thus the second derivative of the strike price of the call with respect to the strike price is the state price density of the underlying asset $X_T$. This result has significant applications especially to static hedging which is a field of active research. For more details and discussion, see for instance Carr \cite{Carr}, \cite{Carr2} and references therein.

In this article we derive multidimensional version of the Breeden-Litzenberger type representation of the state price density. In other words, if the distribution of the underlying asset $\ov{X}_T$ is absolutely continuous with respect to the Lebesgue measure, we show that then the density can be derived from the partial derivatives of the price of rainbow options with payoff
\begin{equation*}
h_p(\ov{x},\ov{K}, K) = \left( \left(\sum_{i=1}^{n}((x_i-K_i)^{+})^p\right)^{1/p}-K\right)^{+},
\end{equation*}
where $0<p< \infty$ and $\ov{X}_T$ is a vector of asset prices at maturity. 

The benefit of our results is that they cover a wide class of models. In particular, we only assume that at least one pricing measure for $\ov{X}_T$ exists. 
We do not assume that it is unique. Moreover, we consider general underlying assets $\ov{X}_T$. 
Hence our results are valid in models which may be complete or incomplete, or discrete or continuous in time.

\section{Results}

In our market model we will assume deterministic interest rate. Thus, without loss of generality, we may omit the interest in calculations. We denote by $m_n$ the Lebesgue measure on $\R^n$ and we denote it simply by $m$ if there is no danger of confusion. Denote
$\R_{+}^{n}=[0\infty)^n$.

In a general model the law of $\ov{X}_T$ under $\Q$ is not necessarily absolutely continuous with respect to Lebesgue measure $m$ on $\R_{+}^n$.  However, typically the state-price density is absolutely continuous with respect to the Lebesgue and then we have nice representations for it.
It was shown by Breeden and Litzenberger \cite{Breeden} that in $1$-dimensional case the risk-neutral density can be obtained by taking the second derivative of the strike price in the call's price functional. In this section we derive similar result for multidimensional case. 

Let $0<p<\infty$. We define a function $h_p:\R_+^{2n+1}\rightarrow\R_+$ by
$$
h_p(\ov{x},\ov{K}, K) = \left( \left(\sum_{i=1}^{n}((x_i-K_i)^{+})^p\right)^{1/p}-K\right)^{+}.
$$
We will denote by $V^p(\ov{K}, K)$ the corresponding price of European rainbow option with payoff $h_p(\ov{X}_T,\ov{K},K)$. We will also consider
function $h_{\infty}$ given as a limit
$$
h_{\infty} = \lim_{p\rightarrow\infty}h_p,
$$
and the corresponding price given by $V^{\infty}$. We begin with some results on relation between prices $V^p$ with different values of $p$. 
\begin{lma}
Let $(\ov{K},K)\in \R_+^{n+1}$ be arbitrary. For every $\ov{x}\in\R_+^n$ we have
\begin{equation}
h_{\infty}(\ov{x},\ov{K},K) = \left(\max_{i}(x_i-K_i)^{+} -K\right)^{+},
\end{equation}
and
\begin{equation}
V^{\infty}(\ov{K},K) = \lim_{p\rightarrow\infty} V^{p}(\ov{K},K).
\end{equation}
\end{lma}
\qed
\begin{prop}\label{prop: lim_sum}
Let $0<p\leq \infty$. Then we have
\begin{equation}
\label{simple_hinta_limit}
\sum_{i=1}^{n}\ \lim_{\underset{j\neq i}{K_j\to \infty}}V^p(K_1,\ldots,K_n,0)=V^1(K_1,\ldots,K_n,0)
\end{equation}
for all $K_1,\ldots, K_n$. The similar conclusion holds for the corresponding payoff functions with
limits taken pointwise. 
\end{prop}
\begin{proof}
It is easy to see that the statement about the payoff functions holds in the sense that the limits 
are taken pointwise. Let now $K_1 \subset \ldots \subset K_n \subset \ldots$ be a sequence of compact sets such that $K_n\uparrow \R_{+}^n$. Now 
(\ref{simple_hinta_limit}) 
follows by considering expectations $\E[h_p]\mathbf{1}_{\ov{X}_T\in K_n}$ and applying monotone convergence theorem.
\end{proof}

\begin{lma}\label{lm: partials}
Suppose that $\Q << m$ on $[0,\infty)^n$. Then
\[\frac{d\Q}{dm}(K_1,\ldots,K_n)=\frac{\partial^n}{\partial K_1 \dots \partial K_n} \Q\left(\bigwedge_i (X_T^i\leq K_i)\right)\quad m\mathrm{-a.e.}\]
\end{lma}
\begin{proof}
Let $\frac{d\Q}{dm}$ be the Radon-Nikodym derivative of $\Q$.
Observe that 
\begin{multline*}
\Q\left(\bigwedge_i (X_T^i\leq K_i)\right)=\int 1_{X_T^1\leq K_1}\dots 1_{X_T^n\leq K_n}\ \frac{d\Q}{dm}\ dm\\
=\int_{0}^{K_1}\ldots \int_{0}^{K_n} \frac{d\Q}{dm}(x_1,\ldots,x_n)\ dx_n \ldots dx_1.
\end{multline*}
According to Lebesgue's differentiation theorem the right hand side is differentiable 
with respect to $K_1$ for $m_1$-a.e. $K_1\geq 0$. The set of tuples $(K_1 , x_2, \ldots,x_n)$, where this differentiability fails is $m$-null, so that we may disregard it. By proceeding in this manner and differentiating $n$ times altogether we obtain the statement.
\end{proof}

\begin{thm}\label{thm: main}
Suppose that $\Q << m$ on $[0,\infty)^n$. Then 
\[\frac{d\Q}{dm}(K_{1},\ldots,K_n)=\frac{\partial^n}{\partial K_1 \dots \partial K_n}\sum_{i=1}^{n}\frac{\partial}{\partial K_i}V^{\infty}(K_1,\ldots,K_n,0)\quad m\mathrm{-a.e.}\]
\end{thm}
\begin{proof}
Denote by $\boldsymbol{1}=(1,1,\ldots,1,0)\in \R^{n+1}$. Here we apply the elementary fact that in the case with continuous partials
the directional derivative can be calculated by taking the inner product of a gradient and a direction vector.
It is easy to see that the above partials are continuous in an open subset of the state space $\R_{+}^n$ with $m$-null complement.
The directional derivative of the payoff function satisfies
\[D_{\boldsymbol{1}} h_{\infty}(X_T^1,\ldots, X_T^n, K_1,\ldots,K_n,0)=-\max_i 1_{X_T^i> K_i}\]
when defined. This limit on the left hand side is both defined and uniform on compact subsets of 
\[\neg\bigwedge_i (X_T^i = K_i)\]
which has clearly $\Q$-null complement. By the uniform convergence and the fact that $\Q$ is a Radon measure we get
\begin{multline*}
\E (-\max_i 1_{X_T^i> K_i})=\E (D_{\boldsymbol{1}} h_{\infty}(\ov{X}_T,K_1,\ldots,K_n,0))\\
=D_{\boldsymbol{1}} \E (h_{\infty}(\ov{X}_T,K_1,\ldots,K_n,0))=D_{\boldsymbol{1}} 
V^{\infty}(K_1,\ldots,K_n,0).
\end{multline*}
On the other hand, $1-\E (\max_i 1_{X_T^i> K_i})=\Q\left(\bigwedge_i (X_T^i\leq K_i)\right)$.
The argument is finished by Lemma \ref{lm: partials}.
\end{proof}

\begin{prop}\label{prop: partial}
The following equality holds
\[\frac{\partial^+}{\partial K_j}V^1(K_1,\ldots,K_n,0)=-\Q(X_T^j\geq K_j).\]
Similarly,
\[\frac{\partial^+}{\partial K_j}h_1(X_T^1,\ldots,X_T^n,K_1,\ldots,K_n,0)=-1_{X_T^j\geq K_j}.\] 
\end{prop}
\begin{proof}
The proof is similar to that of Theorem \ref{thm: main}.
It is easy to see that the right derivative in the latter statement coincides with $-1_{X_T^j\geq K_j}$. 
In fact,
\[\frac{h_1(x_1,\ldots,x_n,K_1,\ldots,K_j+\epsilon,\ldots, K_n,0)- h_1(x_1, \ldots, x_n, K_1,\ldots K_n,0)}{\epsilon} \]
tends to $-1_{X_T^j \geq K_j}$ as $\epsilon\to 0^+$ and even uniformly so in compact subsets of $\neg(X_T^j = K_j)$. Therefore 
\begin{multline*}
-\Q(X_T^j\geq K_j)=\E (-1_{X_T^j\geq K_j})=\E\left(\frac{\partial^+}{\partial K_j}h_1(X_T^1,\ldots,X_T^n,K_1,\ldots,K_n,0)\right)\\
=\frac{\partial^+}{\partial K_j}\E (h_1(X_T^1,\ldots,X_T^n,K_1,\ldots,K_n,0))=\frac{\partial^+}{\partial K_j}V^1(K_1,\ldots,K_n,0).
\end{multline*}

\end{proof}

\begin{thm}\label{thm: Cp}
Suppose that $\Q<<m$ on $[0,\infty)^n$. Then 
\[\frac{d\Q}{dm}(K_1,\ldots,K_n)=\lim_{K\to 0^{+}}\frac{\partial^{n+1}}{\partial K_1 \dots \partial K_n \partial K} V^p(K_1,\ldots,K_n,K)\quad m\mathrm{-a.e.}\]
for $0< p\leq \infty$.
\end{thm}
\begin{proof}
We observe that for a given $K$ the limit
\[\frac{\partial}{\partial K} h_p(X_T^1,\ldots,X_T^n, K_1,\ldots,K_n,K)=-1_{h_p(\ov{X}_T,\ov{K},K) >0}\]
exists and is uniform on compact subsets of 
\[\neg(h_p(\ov{X}_T,\ov{K},0)=K)\]
whose complement is $m$-null. Hence, 
\begin{multline*}
\frac{\partial}{\partial K} \E (h_p (X_T^1,\ldots,X_T^n,K_1,\ldots,K_n,K))\\
=-\Q(h_p (X_T^1,\ldots,X_T^n,K_1,\ldots,K_n,0)  >K),
\end{multline*}
since $\Q$ is a Radon measure. By using the $\sigma$-additivity of $\Q$ we obtain that 
$\lim_{K\to 0^+} \Q(h_p(\ov{X}_T,\ov{K},K) >0)=\Q(\lim_{K\to 0^+} h_p(\ov{X}_T,\ov{K},K) >0)$ exists for $(K_1,\ldots,K_n)$. Thus, by keeping the 
definition of $h_p$ in mind, it is easy to see that 
\[\lim_{K\to 0^{+}}\frac{\partial}{\partial K} V^p(\ov{K},K)=-\Q\left(\bigvee_i (X_T^i > K_i)\right)=\Q\left(\bigwedge_i (X_T^i \leq K_i)\right)-1.\]
The argument is finished similarly as in the proof of Theorem \ref{thm: main} and the order of taking the 
limit can be changed according to the monotone convergence theorem.
\end{proof}

We note that for $n=2$ and $p=1$ the above state price density can be expressed in an alternative form due to the fact that 
\[D_{(-1,-1,1)}V^{1}(K_{1},K_{2},K)=-\Q(X_T^1\geq K_1\wedge X_T^2 \geq K_2).\]
That is,
\[\frac{d\Q}{dm}(K_1,K_2)=\frac{\partial^2}{\partial K_1 \partial K_2}(\frac{\partial V^1}{\partial K_1}+\frac{\partial V^1}{\partial K_2}- \frac{\partial V^1}{\partial K})\quad m\mathrm{-a.e.}\]

\end{document}